\documentclass[12pt,psamsfonts, draft]{amsart}
\usepackage{amssymb,amsmath}
\usepackage{amsthm}
\usepackage{mathrsfs}
\usepackage{enumerate, indentfirst}
\usepackage[fleqn,tbtags]{mathtools}
\usepackage{color}
\usepackage{courier}
\usepackage{hyperref}
\usepackage[a4paper,left=2.5cm,right=2.5cm,top=3cm,bottom=3cm]{geometry}
\usepackage{multicol}
\usepackage{xr}

\newtheorem{theorem}{Theorem}[section]

\newtheorem{lemma}[theorem]{Lemma}

\newtheorem{corollary}[theorem]{Corollary}

\newtheorem{remark}[theorem]{Remark}

\renewcommand{\phi}{\varphi}
\renewcommand{\theta}{\vartheta}

\DeclareMathOperator{\dist}{dist}

\newcommand{\alg}{\mathscr{R}}

\newcommand{\abs}[1]{\lvert#1\rvert}

\newcommand{\nen}{n\in\mathbb{N}}

\newcommand{\ran}{\operatorname{ran}}

\newcommand{\D}{\mathscr{D}}
\newcommand{\M}{\mathscr{M}}

\newcommand{\Hh}{\mathscr{H}}

\newcommand{\Zz}{\mathfrak{Z}}
\newcommand{\Mm}{\mathscr{M}}

\DeclarePairedDelimiterX\sip[2]{(}{)}{#1\,\delimsize\vert\,#2}
\DeclarePairedDelimiterX\siptilde[2]{(}{)_{\!_{\widetilde{A}}}}{#1\,\delimsize\vert\,#2}
\DeclarePairedDelimiterX\sipn[2]{(}{)_{\nu}}{#1\,\delimsize\vert\,#2}
\DeclarePairedDelimiterX\sipm[2]{(}{)_{\mu}}{#1\,\delimsize\vert\,#2}
\DeclarePairedDelimiterX\set[2]{\{}{\}}{#1\,\delimsize\vert\,#2}
\DeclarePairedDelimiterX\dual[2]{\langle}{\rangle}{#1,#2}
\DeclarePairedDelimiterX\sipa[2]{(}{)_{\!_A}}{#1\,\delimsize\vert\,#2}
\DeclarePairedDelimiterX\sipb[2]{(}{)_{\!_B}}{#1\,\delimsize\vert\,#2}

\newcommand{\bph}{\mathbf{B}_+(\mathscr{H})}

\newcommand{\Aa}{\mathscr{A}}

\newcommand{\Ee}{\mathcal{E}}

\newcommand{\dwt}{\mathbf{D}_{\wf}\tf}

\newcommand{\ffi}{\varphi}
\newcommand{\uf}{\mathfrak{u}}

\newcommand{\tf} {\mathfrak{t}}
\newcommand{\wf} {\mathfrak{w}}
\newcommand{\ssf} {\mathfrak{s}}
\newcommand{\Xx}{\mathfrak{X}}
\newcommand{\Yy}{\mathfrak{Y}}

\newcommand{\Ht}{\mathscr{H}_{\tf}}
\newcommand{\Rr}{\mathcal{R}}

\newcommand{\fpx}{\mathcal{F}_+(\Xx)}
\renewcommand*{\phi}{\varphi}
\renewcommand*{\theta}{\vartheta}
\newcommand{\tkerw}{\tf_{{}_{\ker\wf}}}
\newcommand{\ty}{\mathfrak{t}_{\mathfrak{Y}}}
\DeclareMathOperator{\tform}{\mathfrak{t}}

\DeclareMathOperator{\mform}{\mathfrak{t}_{\mu}}
\DeclareMathOperator{\nform}{\mathfrak{t}_{\nu}}

\begin{document}
\title{A Short-type Decomposition Of Forms}

\author[Z. Sebesty\'en]{Zolt\'an Sebesty\'en}
\author[Zs. Tarcsay]{Zsigmond Tarcsay}
\author[T. Titkos]{Tam\'as Titkos}
\address{Z. Sebesty\'en, Zs. Tarcsay, T. Titkos - Institute of Mathematics, E\"otv\"os L. University, P\'azm\'any P\'eter s\'et\'any 1/c., Budapest H-1117, Hungary; }
\email{sebesty@cs.elte.hu,~tarcsay@cs.elte.hu,~titkos@cs.elte.hu}


\begin{abstract} The main purpose of this paper is to present a decomposition theorem for nonnegative sesquilinear forms. The key notion is the short of a form to a linear subspace. This is a generalization of the well-known operator short defined by M. G. Krein. A decomposition of a form into a shorted part and a singular part (with respect to an other form) will be called short-type decomposition. As applications, we present some analogous results for bounded positive operators acting on a Hilbert space; for additive set functions on a ring of sets; and for representable positive functionals on a $\ast$-algebra.
\end{abstract}

\maketitle

\section*{Introduction}

To begin with we give a brief survey of the required definitions and results from \cite{lebdec}, which is our constant reference where the omitted details of this section can be found.

Let $\Xx$ be a complex linear space and let $\tf$ be nonnegative sesquilinear form on it. That is, $\tf$ is a mapping from the Cartesian product
$\Xx\times\Xx$ to $\mathbb{C}$, which is linear in the first argument, antilinear in the second argument, and the corresponding
quadratic form $\tf[\cdot]:\Xx\to\mathbb{R}$
\begin{align*}
\forall x\in\Xx:\quad\tf[x]:=\tf(x,x)
\end{align*}
is nonnegative. In this paper all sesquilinear forms are assumed to be nonnegative, hence we write shortly \emph{form}. The quadratic form of a form fulfills the parallelogram law
\begin{align*}
\forall x,y\in\Xx:\quad\tf[x+y]+\tf[x-y]=2(\tf[x]+\tf[y]).
\end{align*}
According to the Jordan-von Neumann theorem \cite[Satz 1.3]{weidmann}, a form is uniquely determined via its quadratic form, namely
\begin{align*}
\forall x,y\in\Xx:\quad\tf(x,y)=\frac{1}{4}\sum\limits_{k=0}^3i^k\tf[x+i^ky].
\end{align*}
The set $\fpx$ of forms is partially ordered with respect to the ordering
\begin{align*}
\tf\leq\wf~~~\Longleftrightarrow~~~\forall x\in\Xx:\quad\tf[x]\leq\wf[x].
\end{align*}
If there exists a constant $c$ such that $\tf\leq c\cdot\wf$ then we say that $\tf$ is dominated by $\wf$ ($\tf\leq_{\mathrm{d}}\wf$, in symbols). Since the square root of the quadratic form defines a seminorm on $\Xx$, then the kernel of $\tf$
\begin{align*}
\ker\tf:=\big\{x\in\Xx~\big|~\tf[x]=0\big\}
\end{align*}
is a linear subspace of $\Xx$. The Hilbert space $\mathscr{H}_{\tf}$ denotes the completion of the inner product space $\Xx/_{\ker\tf}$ equipped with the natural inner product
\begin{align*}
\forall x,y\in\Xx:\quad(x+\ker\tf~|~y+\ker\tf)_\tf:=\tf(x,y).
\end{align*}
We say that the form $\tf$ is \emph{strongly $\wf$-absolutely continuous} ($\tf\ll_{\mathrm{s}}\wf$), if
\begin{align*}
\forall(x_n)_{n\in\mathbb{N}}\in\Xx^{\mathbb{N}}:\quad\big((\tf[x_n-x_m]\to0)~\wedge~(\wf[x_n]\to0)\big)~\Longrightarrow~\tf[x_n]\to0.
\end{align*}
Remark that this notion is called \emph{closability} in \cite{lebdec}; cf. also \cite{simon}. 
The \emph{singularity} of $\tf$ and $\wf$ (denoted by $\tf\perp\wf$) means that
\begin{align*}
\forall\ssf\in\fpx:\quad\big((\ssf\leq\tf)~\wedge~(\ssf\leq\wf)\big)~\Longrightarrow~\ssf=0.
\end{align*}
The parallel sum $\tf:\wf$ of $\tf$ and $\wf$, and the strongly absolutely continuous (or closable) part $\dwt$ of $\tf$ with respect to $\wf$ are defined by
\begin{align*}
\forall x\in\Xx:\quad(\tf:\wf)[x]:=\inf_{y\in\Xx}\big\{\tf[x-y]+\wf[y]\big\}
    \end{align*}and
\begin{align*}
\dwt:=\sup\limits_{n\in\mathbb{N}}(\tf:n\wf).
\end{align*}
The following decomposition theorem of S. Hassi, Z. Sebesty\'en, and H. de Snoo generalizes the operator decomposition of T. Ando \cite{andolebdec,tarcsayopdec}, the Lebesgue decomposition of finitely additive set functions \cite{darst} (see also \cite{stt,tarcsaylebdec,titkos}), and the canonical decomposition of densely defined forms \cite{simon}.

\begin{theorem}\label{LDT}
Let $\tf$ and $\wf$ be forms on the complex linear space $\Xx$. Then the decomposition
\begin{align*}
\tf=\dwt+(\tf-\dwt)
\end{align*}
is a $(\ll_{\mathrm{s}},\perp)$-type decomposition of the form $\tf$ with respect to $\wf$. That is, $\dwt$ is strongly $\wf$-absolutely continuous, $(\tf-\dwt)$ is $\wf$-singular. Furthermore, this decomposition is extremal in the following sense:
\begin{align*}
\forall\ssf\in\fpx:\quad\big((\ssf\leq\tf)~\wedge~(\ssf\ll_{\mathrm{s}}\wf)\big)~\Longrightarrow~\ssf\leq\dwt.
\end{align*}
%
The decomposition is unique precisely when $\dwt$ is dominated by $\wf$.
\end{theorem}
For the proof see \cite[Theorem 2.11, Theorem 3.8, Theorem 4.6]{lebdec} or  \cite[Theorem 2.3]{stt}.\\

It is a natural idea to consider the following notion of absolute continuity: we say that $\tf$ is \emph{$\wf$-absolutely continuous} ($\tf\ll\wf$) if $\ker\wf\subseteq\ker\tf$, that is to say,
\begin{align*}
\forall x\in\Xx:\quad\wf[x]=0~\Longrightarrow~\tf[x]=0
\end{align*}
in analogy with the well-known measure case. 

The setup of this paper is the following. Our main purpose is to present an $(\ll,\perp)$-type  decomposition theorem for forms which we shall call a \emph{short-type} decomposition. More precisely, for every pair of forms $\tf$ and $\wf$ we shall show that $\tf$ splits into absolutely continuous and singular parts with respect to $\wf$, where the absolutely continuous part is extremal in a certain sense. This will be done in Section \ref{main theorem}. The key notion is the \emph{short of a form}, which is motivated by \cite[Theorem 6]{anderson} of W. N. Anderson and G. E. Trapp.

In Section \ref{operators} we shall see that this is a generalization of the well-known operator short defined by M. G. Krein \cite{krein}. Moreover, we present a factor decomposition for the shorted operator.
As an application, we gain also a short-type decomposition on the set of bounded positive operators (analogous results for matrices can be found in \cite{andersonshort,rosenberg}). That is, for every $A,B\in\bph$ there exist $S,T\in\bph$ such that
\begin{align*}
A=S+T,
\end{align*}where
\begin{align*}
\ran S\subseteq\overline{\ran B}\hspace{0.5cm} \mbox{and}\hspace{0.5cm}\ran T^{1/2}\cap\mathrm{ran}B^{1/2}=\{0\}.
\end{align*}
Furthermore, we prove the following characterization: the range of the bounded positive operator $B$ is closed if and only if for every $A\in\bph$ the short-type decomposition above is unique. In this case, the shorted part of $A$ is closable with respect to $B$.

Another important application can be found in Section \ref{set functions}. Using our main result, we will prove a decomposition theorem for additive set functions. In the $\sigma$-additive case this decomposition coincides with the well-known Lebesgue decomposition of measures, but in the finitely additive case it differs from the Lebesgue-Darst decomposition \cite{darst}. This fact will demonstrate that the Lebesgue-type decomposition, and the short-type decomposition are different in general, and hence, absolute continuity does not implies strong absolute continuity (see also \cite[Example 2]{gudder}).

Finally, in Section \ref{functionals}, we apply our result to present a short-type decomposition for representable positive functionals of a ${}^\ast$-algebra. We emphasize that we do not make any assumptions for the algebra, neither the commutativity, nor the existence of unit element.

\section{A short-type decomposition theorem for forms}\label{main theorem}
Let $\tf$ and $\wf$ be forms on the complex linear space $\Xx$. The purpose of this section is to show that $\tf$ has a decomposition into a $\wf$-absolutely continuous and a $\wf$-singular part. This type decomposition will be called \emph{short-type decomposition}, or $(\ll,\perp)$-type decomposition. In our further considerations an essential role will be played by the concept of the short of a form, which is introduced in the following lemma.
\begin{lemma} Let $\Yy\subseteq\Xx$ be a linear subspace, and let $\tf\in\fpx$. Then the following formula defines a form on $\Xx$
\begin{align*}
\forall x\in\Xx:\quad\tf_{{}_{\Yy}}[x]:=\inf\limits_{y\in\Yy}\tf[x-y].
\end{align*}
Furthermore, $\tf_{{}_{\Yy}}$ is the maximum  of the set
\begin{align*}
\big\{\ssf\in\fpx~\big|~(\ssf\leq\tf)~\wedge~(\Yy\subseteq \ker\ssf) \big\}.
\end{align*}
\end{lemma}
\begin{proof}
Let $\Yy_{\tf}$ be the following subspace of $\Ht$
\begin{align*}
\Yy_{\tf}:=\big\{y+\ker\tf~\big|~y\in \Yy\big\}
\end{align*}
and consider the orthogonal projection $P$ from $\Ht$ onto $\overline{\Yy_{\tf}}$ (the closure of $\Yy_{\tf}$). Then for all $x\in\Xx$
\begin{align*}
\big\|(I-P)(x+\ker\tf)\big\|_{\tf}^2=\mathrm{dist}^2(x+\ker\tf,\overline{\Yy_{\tf}})=\inf\limits_{y\in\Yy}\big\|(x-y)+\ker\tf\big\|_{\tf}^2=\inf_{y\in\Yy}\tf[x-y].
\end{align*}
Consequently, $\tf_{{}_{\Yy}}$ is a form, indeed, and $\Yy\subseteq\ker\tf_{{}_{\Yy}}$. To show the maximality, assume that the quadratic form of $\ssf$ vanishes on $\Yy$ and $\ssf\leq\tf$. According to the triangle inequality we have
\begin{align*}
\ssf[x]\leq\ssf[x-y]\leq\tf[x-y]
\end{align*}
for all $y\in\Yy$, and hence,
\begin{align*}
\ssf[x]\leq\inf\limits_{y\in\Yy}\tf[x-y]=\tf_{{}_{\Yy}}[x].
\end{align*}
\end{proof}
The form $\tf_{{}_{\Yy}}$ is called the \emph{short of the form $\tf$ to the subspace $\Yy$}.\\
It follows from the definition that if $\tf$ and $\wf$ are forms and $\Yy$ and $\Zz$  are linear subspaces, then
\begin{align*}
\big((\tf\leq\wf)\quad\wedge\quad\Yy\subseteq\Zz\big)\quad\Longrightarrow\quad\tf_{\Zz}\leq\wf_{\Yy}.
\end{align*}
Now, we are in position to state and prove the main result of this section.
\begin{theorem} \label{kerdec}
Let $\tf,\wf\in\fpx$ be forms. Then there exists a $(\ll,\perp)$-type decomposition of $\tf$ with respect to $\wf$. Namely,
\begin{align*}
\tf=\tkerw+(\tf-\tkerw),
\end{align*}
where the first summand is $\wf$-absolutely continuous and the second one is $\wf$-singular.
Furthermore, $\tkerw$ is the maximum of the set
\begin{align*}
\big\{\ssf\in\fpx~\big|~(\ssf\leq\tf)~\wedge~(\ssf\ll\wf) \big\}.
\end{align*}
The decomposition is unique precisely when $\tkerw$ is dominated by $\wf$.
\end{theorem}
\begin{proof}
It follows from the previous lemma that $\tkerw\ll\wf$, and that $\tkerw$ is maximal. Let $\ssf$ be a form such that $\ssf\leq\wf$ and $\ssf\leq\tf-\tkerw$. Since $\tkerw\leq\tkerw+\ssf\leq\tf$ and the quadratic form of $\tkerw+\ssf$ vanishes on $\ker\wf$, the maximality of $\tkerw$ implies that $\ssf=0$.\\
It remains only to prove that the decomposition is unique if and only if $\tkerw$ is dominated by $\wf$. Let $c$ be a constant such that $\tf\leq c\cdot\wf$ (we may assume that $c>1$) and let $\tf=\tf_1+\tf_2$ be an $(\ll,\perp)$-type decomposition. Since $\tkerw$ is maximal, we have
\begin{align*}
\tf_2=\tf-\tf_1\geq\tkerw-\tf_1\geq\frac{1}{c}(\tkerw-\tf_1)\hspace{0.5cm}\mbox{and}\hspace{0.5cm}\wf\geq\frac{1}{c}\tkerw\geq\frac{1}{c}(\tkerw-\tf_1)
\end{align*}
which is a contradiction, unless $\tkerw=\tf_1$.
Finally, observe that $\dwt\leq\tkerw$, and therefore, every $(\ll_{\mathrm{s}},\perp)$-type decomposition is a $(\ll,\perp)$-type decomposition as well. Indeed,
\begin{align*}
\tkerw[x]=\inf_{y\in\ker\wf}\tf[x-y]=\inf_{y\in\ker\wf}\{\tf[x-y]+n\wf[y]\}\geq\inf_{y\in\Xx}\{\tf[x-y]+n\wf[y]\}=(\tf:n\wf)[x]
\end{align*}
holds for all $n\in\mathbb{N}$ and $x\in\Xx$, therefore
\begin{align*}
\dwt=\sup\limits_{n\in\mathbb{N}}(\tf:n\wf)\leq\tkerw.
\end{align*}
Thus if the $(\ll,\perp)$-type decomposition is unique, then $\tkerw=\dwt$, and $\tkerw\leq_{\mathrm{d}}\wf$ according to Theorem \ref{LDT}.
\end{proof}
Observe that $(\ty)_{\Yy}=\ty$ for each subspace $\Yy$, i.e., shortening to a subspace is an idempotent operation. Furthermore, $\tf\ll\wf$ precisely when $\tkerw=\tf$.

\begin{remark} Let $\Aa$ be a complex algebra, let $\mathscr{I}\subseteq\Aa$ be a left ideal, and let $\tf$ be a \emph{representable form} on $\Aa$. That is, a nonnegative sesquilinear form, which satisfies
\begin{align*}
(\forall a\in\Aa)~(\exists \lambda_a>0)~(\forall b\in \Aa):\quad\tf[ab]\leq\lambda_a\tf[b].
\end{align*}
A simple observation shows that $\tf_{\mathscr{I}}$ is representable
\begin{align*}
\tf_{\mathscr{I}}[ab]=\inf_{x\in\mathscr{I}}\tf[ab-x]\leq\inf_{x\in\mathscr{I}}\tf[ab-ax]
\leq\inf_{x\in\mathscr{I}}\lambda_a\tf[b-x]=\lambda_a\tf_{\mathscr{I}}[b].
\end{align*}
If $\wf$ is a representable form on $\Aa$ as well, then $\ker\wf$ is obviously a left ideal, and hence we have the following decomposition
\begin{align*}\tf=\tf_{{}_{\ker\wf}}+(\tf-\tf_{{}_{\ker\wf}})
\end{align*}
where $\tf_{{}_{\ker\wf}}\ll\wf$, $(\tf-\tf_{{}_{\ker\wf}})\perp\wf$, and $\tf_{{}_{\ker\wf}}$ is representable. For a Lebesgue-type decomposition of representable forms we refer the reader to \cite{jot}.
\end{remark}

Finally, we show that the shorted form $\ty$ possesses an extremal property. In fact, we prove that $\ty$ is a disjoint part of $\tf$ for every subspace $\Yy$, or equivalently, $\ty$ is a so-called $\tf$-quasi unit. After recalling the corresponding definitions, in Lemma \ref{t-q-u char} we give a characterization of the extremal points of the convex set
\begin{align*}[0,\tf]=\big\{\wf\in\fpx~\big|~\wf\leq\tf\big\}.
\end{align*}

We say that $\uf$ is a \emph{$\tf$-quasi-unit}, if $\mathbf{D}_{\uf}\tf=\uf$. The form $\uf$ is a \emph{disjoint part} of $\tf$ if $\uf$ and $\tf-\uf$ are singular. The set of extremal points of a convex set $C$ is denoted by $\operatorname{ex} C$. For the terminology see \cite{el,riesz,cof}.

\begin{lemma}\label{t-q-u char} Let $\tf$ and $\uf$ be forms on $\D$ such that $\uf\leq\tf$, and let $\lambda>0$ and $\mu>0$ be arbitrary real numbers. Then the following statements are equivalent.
\begin{multicols}{2}
\begin{itemize}
\item[$(i)$] $\uf$ is a $\tf$-quasi-unit, i.e., $\mathbf{D}_{\uf}\tf=\uf$.
\item[$(ii)$] There exists $\wf$ such that $\uf=\dwt$.
\item[$(iii)$] $\uf$ is a disjoint part of $\tf$.
\item[$(iv)$] $\uf\in\mathrm{ex}[0,\tf]$.
\item[$(v)$] $(\lambda\uf):(\mu\tf)=\frac{\lambda\mu}{\lambda+\mu}\uf$.
\item[$(vi)$] $(\lambda\uf):\tf=\uf:(\lambda\tf)$.
\end{itemize}
\end{multicols}
\end{lemma}

\begin{proof}
Here we prove only $(i)\Rightarrow (v)\Rightarrow (vi)\Rightarrow (i)$. The remainder can be found in \cite[Theorem 11]{cof}. Assume that $\uf$ is a $\tf$-quasi unit, and observe that
\begin{align*}
(\lambda\uf):(\mu\tf)=(\lambda\uf):\big(\mathbf{D}_{\lambda\uf}(\mu\tf)\big)=(\lambda\uf):(\mu\mathbf{D}_{\uf}\tf)=(\lambda\uf):(\mu\uf)=\frac{\lambda\mu}{\lambda+\mu}\uf.
\end{align*}
according to the properties of the parallel sum and the following equalities
\begin{align*}
\tf:\wf=\mathbf{D}_{\wf}(\tf:\wf)=\dwt:\wf
\end{align*}
(see \cite[Lemma 2.3, Lemma 2.4, Proposition 2.7]{lebdec}). Assuming $(v)$ it is clear that
\begin{align*}
(\lambda\uf):\tf=\frac{\lambda}{1+\lambda}\uf=\uf:(\lambda\tf).
\end{align*}
Finally, since $\uf\leq\tf$, property $(vi)$ implies that
\begin{align*}
\mathbf{D}_{\uf}\tf=\sup\limits_{\nen}\big(\tf:(n\uf)\big)=\sup\limits_{\nen}\big((n\tf):\uf\big)=\mathbf{D}_{\tf}\uf=\uf.
\end{align*}
\end{proof}

\begin{theorem}\label{shortext}
Let $\tf$ be a form on $\Xx$ and let $\Yy$ be a linear subspace of $\Xx$. Then $\tf_{{}_{\Yy}}$ is an extremal point of the convex set \begin{align*}\{\ssf\in\fpx~|~0\leq\ssf\leq\tf\}
\end{align*}
and
\begin{align*}
\mathbf{D}_{\tf_{{}_{\Yy}}}\tf=\tf_{{}_{\Yy}}.
\end{align*}
\end{theorem}
\begin{proof} According to the previous lemma, it is enough to show that $\ty$ is a disjoint part of $\tf$. That is, $\ty$ and $\tf-\ty$ are singular. Let $\ssf$ be a form such that $\ssf\leq\ty$ and $\ssf\leq\tf-\ty$. Then $\ty+\ssf$ vanishes on $\Yy$  and $\ty+\ssf\leq\tf$, thus the maximality of $\ty$ implies that $\ssf=0$.
\end{proof}

\section{Bounded positive operators}\label{operators}
Let $\Hh$ be a complex Hilbert space with the inner product $(\cdot~|~\cdot)$ and the norm $\|\cdot\|$. The set of bounded positive operators will be denoted by $\bph$. The notation $A\leq B$ stands for the usual relation
\begin{align*}
\forall x\in\Hh:\quad(Ax~|~x)\leq(Bx~|~x).
\end{align*}
For every $A\in\bph$ we set
\begin{align*}
\forall x,y\in\Hh:\quad\tf_{{}_A}(x,y):=(Ax~|~y)
\end{align*}
which defines a bounded nonnegative form on $\mathscr{H}$. Conversely, in view of the Riesz-representation theorem, the correspondence $A\mapsto\tf_{{}_A}$ defines a bijection between bounded positive operators and bounded nonnegative forms. Consequently, we can define the domination, (strong) absolute continuity, and singularity analogously to the ones defined for forms. We write $A\leq_{\mathrm{d}}B$ if there exists a constant $c$ such that $A\leq c\cdot B$. If $Bx=0$ implies that $Ax=0$ for all $x\in\Hh$, we say that $A$ is \emph{$B$-absolutely continuous} ($A\ll B$). The operators $A$ and $B$ are \emph{singular} ($A\perp B)$ if $0$ is the only positive operator which is dominated by both $A$ and $B$. Finally, $A$ is \emph{strongly $B$-absolutely continuous} ($A\ll_{\mathrm{s}} B$) if for any sequence $(x_n)_{n\in\mathbb{N}}\in\Hh^{\mathbb{N}}$
\begin{align*}
\big((A(x_n-x_m)~|~x_n-x_m)\to0~\wedge~(Bx_n~|~x_n)\to0\big)\quad\Rightarrow\quad(Ax_n~|~x_n)\to0.
\end{align*}
Remark that
\begin{align*}
A\ll B~\Longleftrightarrow~\ker B\subseteq \ker A\hspace{0.5cm}\mbox{and}\hspace{0.5cm}A\perp B~\Longleftrightarrow~\ran A^{1/2}\cap\ran B^{1/2}=\{0\},
\end{align*}
see \cite{andolebdec} or \cite{tarcsayopdec}.
It was proved by Krein that if $\Mm$ is a closed linear subspace of $\Hh$ and $A\in\bph$, then the set
\begin{align*}
\big\{S\in\bph~\big|~(S\leq A)~\wedge~(\ran S\subseteq\Mm)\big\}
\end{align*}
possesses a greatest element. This follows immediately from our previous results, and this is why we say that the form $\tf_{{}_{\Yy}}$ is the \emph{short of} $\tf$ \emph{to the subspace} $\Yy$. Indeed, let $\tf(x,y)=(Ax~|~y)$ and consider the form $\tf_{\Mm^{\perp}}$. Since $\tf_{\Mm^{\perp}}$ is a bounded form, there exists a unique $S\in\bph$ such that $\tf_{\Mm^{\perp}}(x,y)=(Sx~|~y)$ and
\begin{align*}
x\in\Mm^{\perp}\Longrightarrow\tf_{\Mm^{\perp}}[x]=0\Longrightarrow(Sx~|~x)=0\Longrightarrow\Mm^{\perp}\subseteq\ker S\Longrightarrow \ran S\subseteq\Mm.
\end{align*}
The maximality of $S$ follows from the maximality of $\tf_{\Mm^{\perp}}$. Now, since the map $A\mapsto\tf_{{}_A}$ is an order preserving positive homogeneous map from $\bph$ into $\mathcal{F}_+(\mathscr{H})$, the following theorem is an immediate consequence of Theorem \ref{kerdec}.

\begin{theorem}\label{opdec}
Let $A$ and $B$ be bounded positive operators on $\mathscr{H}$. Then there is a decomposition of $A$ with respect to $B$ into $B$-absolutely continuous and $B$-singular parts. Namely,
\begin{align*}
A=A_{{}_{\ll,B}}+A_{{}_{{\perp},B}}.
\end{align*}
The decomposition is unique, precisely when $A_{{}_{\ll,B}}$ is dominated by $B$.
\end{theorem}
\begin{proof} Let $A_{{}_{\ll,B}}$ and $A_{{}_{{\perp},B}}$ be the operators corresponding to $(\tf_{{}_A})_{\ker\tf_B}$ and $\tf_{{}_A}-(\tf_{{}_A})_{\ker\tf_B}$, respectively.
\end{proof}
\begin{corollary} Let $B$ be a bounded positive operator with closed range. Then for every $A\in\bph$
\begin{align*}
A=A_{{}_{\ll,B}}+A_{{}_{{\perp},B}}.
\end{align*}
is the unique decomposition of $A$ into $B$-absolutely continuous and $B$-singular parts.
\end{corollary}
\begin{proof}
If $\ran B$ is closed, then the following two sets are identical according to the well-known theorem of Douglas \cite{douglas}
\begin{align*}
\big\{S\in\bph~\big|~(S\leq A)~\wedge~(\ran S\subseteq\ran B)\big\}=\big\{S\in\bph~\big|~(S\leq A)~\wedge~(S\leq_{\mathrm{d}} B)\big\}.
\end{align*}
Consequently, the uniqueness follows from Theorem \ref{opdec}. Since $\ran B$ is closed, the inclusion $\ker B\subseteq\ker A_{{}_{\ll,B}}$ implies that $\ran A_{{}_{\ll,B}}\subseteq\ran B$.
\end{proof}
Observe that if $\ran B$ is closed, then $A_{{}_{\ll,B}}$ coincides with $\mathbf{D}_BA$ in the sense of Ando \cite{andolebdec}, and therefore it is strongly absolutely continuous (or closable) with respect to $B$. Furthermore, according to \cite[Theorem 7]{tarcsayopdec} we have the following characterization of closed range positive operators.
\begin{corollary} Let $B$ be a bounded positive operator. Then the following are equivalent
\begin{itemize}
\item[$(i)$] $\ran B$ is closed,
\item[$(ii)$] $\forall A\in\bph:\quad A_{{}_{\ll,B}}\leq_{\mathrm{d}}B$,
\item[$(iii)$] $\forall A\in\bph:\quad \mathbf{D}_BA\leq_{\mathrm{d}}B$.
\end{itemize}
If any of $(i)-(iii)$ fulfills, then $\mathbf{D}_BA=A_{{}_{\ll,B}}$ for all $A\in\bph$.
\end{corollary}

\begin{corollary}
Let $A$ be a bounded positive operator. Then $A_{{}_{\ll,B}}$ is an extremal point of the operator segment
\begin{align*}
[0,A]:=\{S\in\bph~|~S\leq A\}
\end{align*}
for all $B\in\bph$.
\end{corollary}
 We remark that the short $A_{\M}$ of $A$ to the close linear subspace $\M$ of the  (complex) Hilbert space $\Hh$ possesses a factorization of the form 
 \begin{equation*}
    A_{\M}=A^{1/2}P_{\widetilde{\M}}A^{1/2},
 \end{equation*}
 where $P_{\widetilde{\M}}$ is defined to be the orthogonal projection onto the subspace $\widetilde{\M}:=A^{-1/2}\langle\M\rangle$, see Krein \cite{krein}. This factorization can hold, of course, only if the underlying space is complex. Below we offer an alternative factorization of the operator short that simultaneously treats the real and complex cases. In fact, we show that there exists a (real or complex, respectively) Hilbert space $\Hh_A$, associated with the positive operator $A$, such that $A_{\M}$ admits a factorization of the form $J_A(I-P)J_A^*$ where $J_A$ is the canonical continuous embedding of $\Hh_A$ into $\Hh$ and $P$ is the orthogonal projection onto an appropriately defined subspace of $\Hh_A$, associated with $\M$. The construction below is taken from \cite{Sebestyén93}.

Let us consider the range space $\ran A$, equipped with the inner product $\sipa{\cdot}{\cdot}$
\begin{equation*}
    \forall x,y\in\Hh:\qquad \sipa{Ax}{Ay}=\sip{Ax}{y}.
\end{equation*}
Note that the operator Schwarz inequality
\begin{equation*}
    \sip{Ax}{Ax}\leq\|A\|\sip{Ax}{x}
\end{equation*}
implies that $\sipa{\cdot}{\cdot}$ defines an inner product, indeed. Let $\Hh_A$ stand for the completion of that inner product space. Consider the canonical embedding operator of $\ran A\subseteq \Hh_A$ into $\Hh$, defined by
\begin{equation*}
    \forall x\in\Hh:\qquad J_A(Ax):=Ax.
\end{equation*}
Then $J_A$ is well defined and continuous due to the operator Schwarz inequality above (namely, by norm bound $\sqrt{\|A\|}$). This mapping has a unique norm preserving extension from $\Hh_A$ to $\Hh$ which is denoted by $J_A$ as well. An easy calculation shows that its adjoint $J_A^*$ acts as an operator from $\Hh$ to $\Hh_A$ possessing the canonical property
\begin{equation*}
    \forall x\in\Hh:\qquad J_A^*x=Ax.
\end{equation*}
This yields the following useful factorization for $A$:
\begin{equation*}
    A=J_AJ_A^*.
\end{equation*}
\begin{theorem}\label{factorthm}
    Let $\Hh$ be a Hilbert space and let $A\in\bph$. For a given subspace $\M\subseteq \Hh$ denote by $P$ the orthogonal projection of $\Hh_A$ onto   the closure of  $\{Ax\,|\,x\in\M\}$. Then the short of $A$ to $\M$ equals $J_A(I-P)J_A^*$.
\end{theorem}
\begin{proof}
    It is enough to show that the quadratic forms of $J_A(I-P)J_A^*$ and $\tf_{\M^{\perp}}$ are equal. To verify this let $x\in\Hh$. Then
    \begin{align*}
        \sip{J_A(I-P)J_A^*x}{x}&=\sipa{(I-P)Ax}{(I-P)Ax}=\dist^2(Ax;\ran P)\\
        &=\inf_{y\in\M} \sipa{Ax-Ay}{Ax-Ay}=\inf_{y\in\M} \sip{A(x-y)}{x-y}\\
        &=\tf_{\M^{\perp}}[x],
    \end{align*}
    as it is claimed.
\end{proof}

The above construction yields another formula for the quadratic form of the shorted operator:
\begin{corollary}
    Let $\Hh$ be a Hilbert space, $A\in\bph$ and $\M\subseteq\Hh$ any closed linear subspace. Then for any $x\in\Hh$
    \begin{equation*}
     \sip{J_A(I-P)J_A^*x}{x}=\sip{Ax}{x}-\sup\{\abs{\sip{Ax}{y}}^2\,|\,y\in\M, \sip{Ay}{y}\leq1\}.
    \end{equation*}
\end{corollary}
\begin{proof}
    For $x\in\Hh$ we have
    \begin{align*}
        \sip{J_A(I-P)J_A^*x}{x}&=\sipa{Ax}{Ax}-\sipa{P(Ax)}{P(Ax)}\\
        &=\sip{Ax}{x}-\sup\{\abs{\sipa{Ax}{Ay}}^2\,|\, y\in\M, \sipa{Ay}{Ay}\leq1\}\\
        &=\sip{Ax}{x}-\sup\{\abs{\sip{Ax}{y}}^2\,|\,y\in\M, \sip{Ay}{y}\leq1\},
    \end{align*}
    indeed.
\end{proof}
\begin{corollary}
    If $A$ and $B$ are bounded positive operators on the Hilbert space $\Hh$ then the quadratic forms of $A_{\ll,B}$ and $A_{\perp,B}$ can be calculated by the following formulae:
    \begin{align*}
    \sip{A_{\ll,B}x}{x}&=\inf_{y\in\ker B} \sip{A(x-y)}{x-y},\\
        \sip{A_{\perp,B}x}{x}&=\sup\{\abs{\sip{Ax}{y}}^2\,|\,y\in\ker B, \sip{Ay}{y}\leq1\}.
    \end{align*}
\end{corollary}
\begin{proof}
    Since $A_{\ll,B}$ is nothing but the short of $A$ to the closed subspace $\ker B^{\perp}$, Theorem \ref{factorthm} together with the above corollary implies the desired formulae.
\end{proof}

\section{Additive set functions}\label{set functions}
In this section we apply our main theorem for finitely additive nonnegative set functions. Our main reference is \cite{stt}. We recall first some definitions.

Let $T$ be a non-empty set, and let $\alg$ be a ring of some subsets of $T$. Let $\mu$ and $\nu$ be (finitely) additive nonnegative set functions (or charges, for short) on $\alg$. We say that $\nu$ is \emph{strongly absolutely continuous} with respect to $\mu$ (in symbols $\nu\ll_{\mathrm{s}}\mu$) if for any $\varepsilon>0$ there exists $\delta>0$ such that $\mu(R)<\delta$ implies
$\nu(R)<\varepsilon$ for all $R\in\alg$. It is important to remark that this notion is referred to as absolutely continuity in \cite{stt}. We say that the charge $\nu$ is \emph{absolutely continuous} with respect to $\mu$ if $\mu(R)=0$ implies $\nu(R)=0$ for all $R\in\alg$. Finally $\nu$ and $\mu$ are singular if the only charge which is dominated by both $\nu$ and $\mu$ is the zero charge.\\

Let $\Ee$ be the complex vector space of $\alg$-step functions, and define the associated form $\nform$ as follows:
\begin{equation*}\label{assoc_forms}
\forall\ffi,\psi\in\Ee:\quad\nform(\ffi,\psi):=\int\limits_T \ffi\cdot\overline{\psi}~\mathrm{d}\nu.
\end{equation*}
It was proved in \cite[Theorem 3.2]{stt} that if $\mu$ and $\nu$ are bounded charges, then $\nu$ is strongly absolutely continuous with respect to $\mu$ if and only if $\nform$ is strongly absolutely continuous with respect to $\mform$. Similarly, $\nu$ and $\mu$ are singular precisely when $\nform$ and $\mform$ are singular.

Using this result, the authors proved the classical Lebesgue-Darst decomposition theorem. Namely, if $\mu$ and $\nu$ are bounded charges then the formula
\begin{align*}
\nu_a:\alg\mapsto\mathbf{D}_{\tf_{\mu}}\tf_{\nu}[\chi_R]
\end{align*}
defines a charge on $\alg$, such that $\nu_a\ll_{\mathrm{s}}\mu$ and $(\nu-\nu_a)\perp\mu$. We use this argument below to provide a $(\ll,\perp)$-type decomposition. The following lemma (see \cite[Lemma 3.3]{stt}) plays an essential role in the proof and may be very useful in deciding the additivity of the correspondence $R\mapsto\tf[\chi_R]$ for a given form $\tf$.

\begin{lemma}
Let $T$ be a non-empty set, and let $\alg$ be a ring of subsets of $T$. For a given form $\tform$ on $\Ee$ the following statements are equivalent:
\begin{enumerate}[\upshape (i)]
  \item The set function $\theta:\alg\to\mathbb{R}$ defined by $\theta(R):=\tf[\chi_R]$ is additive;
  \item $\tform[\zeta]=\tform[|\zeta|]$ for all $\zeta\in\Ee$.
\end{enumerate}
\end{lemma}
The main result of this section is the following short-type decomposition of charges. Here we emphasize that, in contrast to the Lebesgue-Darst decomposition, this decomposition holds for not necessarily bounded charges as well.
\begin{theorem}\label{measure-decomp}
Let $\alg$ be a ring of subsets of a non-empty set $T$, and let $\mu$ and $\nu$ be charges on $\alg$. Then there is a decomposition
\begin{align*}\nu=\nu_{{}_{\ll,\mu}}+\nu_{{}_{\perp,\mu}},
\end{align*}
where $\nu_{{}_{\ll,\mu}}\ll\mu$ and $\nu_{{}_{\perp,\mu}}\perp\mu$. Furthermore, if $\theta$ is a charge such that $\theta\leq\nu$ and $\theta\ll\mu$, then $\theta\leq\nu_{{}_{\ll,\mu}}$.
\end{theorem}
\begin{proof} Let us define the set function $\nu_{{}_{\ll,\mu}}$ by
\begin{align*}
\forall R\in\alg:\quad\nu_{{}_{\ll,\mu}}(R):=(\tf_{\nu})_{\ker\tf_{\mu}}[\chi_R].
\end{align*}
It is clear that $\mu(R)=0$ implies $\nu_{{}_{\ll,\mu}}(R)=0$. Our only claim is therefore to prove the additivity of $\nu_{{}_{\ll,\mu}}$. For this purpose, let $\ffi\in\Ee$. In accordance with the previous lemma, it is enough to show that
\begin{align*}
(\tf_{\nu})_{\ker\tf_{\mu}}[|\ffi|]=(\tf_{\nu})_{\ker\tf_{\mu}}[\ffi].
\end{align*}
Assume that
\begin{align*}
\ffi=\sum\limits_{i=1}^k\lambda_i\cdot\chi_{R_i},
\end{align*}
where $\{\lambda_i\}_{i=1}^k$ are non-zero complex numbers and $\{R_i\}_{i=1}^k$ are pairwise disjoint elements of $\alg$. Define the function $\psi$ as follows
\begin{align*}
\psi:=\sum\limits_{i=1}^k\frac{|\lambda_i|}{\lambda_i}\cdot\chi_{{}_{R_k}}+\chi_{{}_{T\setminus \bigcup_{i=1}^kR_i}}.
\end{align*}
Since $|\psi(t)|=1$ for all $t\in T$, the multiplication with $\psi$ is a bijection on $\Ee$. (Note that $\psi\notin\Ee$ in general.) As $\tf_{\nu}[\zeta]=\tf_{\nu}[|\zeta|]$ for all $\zeta\in\Ee$, we have that
\begin{align*}
\begin{split}
(\tf_{\nu})_{\ker\tf_{\mu}}[\ffi]&=\inf\limits_{\xi\in\Ee}\tf_{\nu}[\ffi-\xi]=\inf\limits_{\xi\in\Ee}\tf_{\nu}[|\ffi-\xi|]\\
&=\inf\limits_{\xi\in\Ee}\tf_{\nu}[|\psi|\cdot|\ffi-\xi|]=\inf\limits_{\xi\in\Ee}\tf_{\nu}[||\ffi|-\psi\cdot\xi|]\\
&=\inf\limits_{\xi\in\Ee}\tf_{\nu}[|\ffi|-\psi\cdot\xi]=(\tf_{\nu})_{\ker\tf_{\mu}}[|\ffi|].
\end{split}
\end{align*}
Consequently, $\nu_{{}_{\ll,\mu}}$ is a charge, which is absolutely continuous with respect to $\mu$. Since $\nu$ and $\nu_{{}_{\ll,\mu}}$ are charges, $\nu_{{}_{\perp,\mu}}:=\nu-\nu_{{}_{\ll,\mu}}$ is a charge too, which is derived from $\tf_{\nu}-(\tf_{\nu})_{\ker\tf_{\mu}}$. Hence, $\nu_{{}_{\perp,\mu}}$ and $\mu$ are singular.
\end{proof}
The following corollary is an immediate consequence of Theorem \ref{shortext}.
\begin{corollary} Let $\nu$ and $\mu$ be a charges on $\Rr$. Then $\nu_{{}_{\ll,\mu}}$ is an extremal point of the convex set of those charges that are majorized by $\nu$.
\end{corollary}

\begin{remark} If $\alg$ is a $\sigma$-algebra, $\mu$ and $\nu$ are $\sigma$-additive (i.e., $\mu$ and $\nu$ are measures), then the notions of absolute continuity and strong absolute continuity coincide, and hence
\begin{align*}
\mathbf{D}_{\tf_{\mu}}\tf_{\nu}=(\tf_{\nu})_{\ker\tf_{\mu}}.
\end{align*}
In this case, the short-type decomposition coincides with the classical Lebesgue decomposition. Furthermore, we have the following formula for the absolutely continuous part
\begin{align*}
\forall R\in\Rr:\quad\nu_{{}_{\ll,\mu}}(R)=\inf\limits_{\ffi\in\Ee}\int\limits_R |1-\ffi(t)|^2~\mathrm{d}\nu(t).
\end{align*}
If $\alg$ is an algebra of sets, and we consider finitely additive charges on it, then the involved absolute continuity concepts are different. Consequently, there exist $\mu$ and $\nu$ such that
\begin{align*}
\mathbf{D}_{\tf_{\mu}}\tf_{\nu}\neq(\tf_{\nu})_{\ker\tf_{\mu}}.
\end{align*}
\end{remark}
\section{Representable functionals}\label{functionals}

The Lebesgue-type decomposition of positive functionals were studied by several authors, see e.g. \cite{gudder, inoue, kosaki, jot, szucs, tpf-arxiv}. Sz\H{u}cs in \cite{jot} proved that the Lebesgue-type decomposition of representable positive functionals can be derived from their induced sesquilinear forms. In this section we present a short-type decomposition for representable positive functionals, which corresponds to the short type decomposition of their induced forms.

Let $\mathscr{A}$ be a complex ${}^{\ast}$-algebra and let $f:\mathscr{A}\to\mathbb{C}$ be a positive linear functional on it (that is, $f(a^\ast a)\geq0$ for all $a\in \mathscr{A}$). The form induced by $f$ will be denoted by $\tf_f$
\begin{align*}
\tf_f:\mathscr{A}\times \mathscr{A}\to\mathbb{C};\quad\quad(a,b)\mapsto f(b^{\ast}a).
\end{align*}
For positive functionals $f\leq g$ means that $\tf_f\leq\tf_g$. The positive functional $f$ is called \emph{representable}, if there exists a Hilbert space $\mathscr{H}_{{}_f}$, a ${}^{\ast}$-representation $\pi_{{}_f}$ of $\mathscr{A}$ into $\mathscr{H}_{{}_f}$, and a cyclic vector $\xi_{{}_f}\in\mathscr{H}_{{}_f}$ such that
\begin{align*}
\forall a\in \mathscr{A}:\quad f(a)=(\pi_{{}_f}(a)\xi_{{}_f}~|~\xi_{{}_f})_{{}_f}.
\end{align*}
Such a triple $(\mathscr{H}_{{}_f},\pi_{{}_f},\xi_{{}_f})$ is provided by the classical GNS-construction (see \cite{sebestyen} for the details): namely, denote by $N_f$ the set of those elements $a$ such that $f(a^\ast a)=0$, and let $\mathscr{H}_f$ stand for the Hilbert space completion of the inner product space
\begin{align*}
\big(\Aa/_{N_f}, (\cdot~|~\cdot)_f\big) ; \quad\quad\forall a,b\in \Aa:\quad (a+N_f~|~b+N_f)_f:=\tf_f(a,b)=f(b^\ast a).
\end{align*}
For $a\in\Aa$ let $\pi_f(a)$ be the left multiplication by $a$:
\begin{align*}
\forall x\in \Aa:\quad\pi_f(a)(x+N_f):=ax+N_f.
\end{align*}
The cyclic vector $\xi_f$ is defined as the Riesz-representing vector of the continuous linear functional
\begin{align*}
\mathscr{H}_f\supseteq\mathscr{A}/_{N_f}\to\mathbb{C} ; \quad\quad a+N_f\mapsto f(a).
\end{align*}
Note also that \begin{align*}\pi_f(a)\xi_f=a+N_f.\end{align*} We define the absolute continuity and singularity as for forms. Singularity means that the zero functional is the only representable functional which is dominated by both $f$ and $g$. According to \cite[Theorem 2]{sing}, this is equivalent with the singularity of the forms $\tf_f$ and $\tf_g$. We say that $f$ is $g$-absolutely continuous ($f\ll g$), if
\begin{align*}
\forall a\in \mathscr{A}:\quad g(a^\ast a)=0~\Longrightarrow~f(a^\ast a)=0.
\end{align*}
A decomposition of $f$ into representable $g$-absolutely continuous and $g$-singular parts is called short-type decomposition.

Now, the short-type decomposition for representable functionals can be stated as follows.
\begin{theorem}Let $f$ and $g$ be representable positive functionals on the ${}^{\ast}$-algebra $\mathscr{A}$. Then $f$ admits a decomposition
\begin{align*}
f=f_{{}_{\ll,g}}+f_{{}_{\perp,g}}
\end{align*}
to a sum of representable functionals, where $f_{{}_{\ll,g}}$ is $g$-absolutely continuous, $f_{{}_{\perp,g}}$ and $g$ are singular. Furthermore, $f_{{}_{\ll,g}}$ is the greatest among all of the representable functionals $h$ such that $h\leq f$ and $h\ll g$.
\end{theorem}
\begin{proof}
Let $\mathscr{M}$ be the following closed subspace of $\mathscr{H}_f$
\begin{align*}
\mathscr{M}:=\overline{\{a+N_f~|~g(a^\ast a)=0\}}
\end{align*}
and let $P$ be the orthogonal projection from $\mathscr{H}_f$ onto $\mathscr{M}$. Then $\mathscr{M}$ and $\mathscr{M}^{\perp}$ are $\pi_f$-invariant subspaces. Since $\pi_f$ is a ${}^{\ast}$-representation, it is enough to prove that $\mathscr{M}$ is $\pi_f$ invariant. Let $a,x\in \mathscr{A}$ and assume that $g(a^\ast a)=0$. Then
\begin{align*}
\pi_f(x)(a+N_f)=xa+N_f\in\mathscr{M}
\end{align*}
because
\begin{align*}
g(a^{\ast} x^{\ast}xa)=\| \pi_g(x)(a+N_f)\|_g^2\leq \| \pi_g(x)\|_g^2\cdot g(a^{\ast} a)=0.
\end{align*}
Consequently,
\begin{align*}
\pi_f(x)\langle\mathscr{M}\rangle\subseteq\overline{\pi_f(x)\langle\{a+N_f~|~g(a^\ast a)=0\}\rangle}\subseteq\mathscr{M},
\end{align*}
as it is stated. Now, let us define the functionals
\begin{align*}
f_{{}_{\ll},g}(a):=(\pi_f(a)(I-P)\xi_f~|~(I-P)\xi_f)_f.
\end{align*}
\begin{align*}
f_{{}_{\perp,g}}(a):=(\pi_f(a)P\xi_f~|~P\xi_f)_f.
\end{align*}
Clearly, $f_{{}_{\ll,g}}$ and $f_{{}_{\perp,g}}$ are representable positive functionals. On the other hand, since $\mathscr{M}$ is $\pi_f$-invariant, we find that
\begin{align*}
f_{{}_{\ll,g}}(a^{\ast}a)=\|\pi_f(a)(I-P)\xi_f\|_f^2=\|(I-P)\pi_f(a)\xi_f\|_f^2=\|(I-P)(a+N_f)\|_f^2=\tf_{f_{{}_{\ll,g}}}[a]
\end{align*}
and similarly,
\begin{align*}
f_{{}_{\perp,g}}(a^{\ast}a)=\|P(a+N_f)\|_f^2=\tf_{f_{{}_{\perp,g}}}[a].
\end{align*}
Since $\tf_{f_{{}_{\ll,g}}}$ is $\tf_g$-absolutely continuous, and $\tf_{f_{{}_{\perp,g}}}$ is $\tf_g$-singular, we infer that $f_{{}_{\ll,g}}\ll g$ and $f_{{}_{\perp,g}}\perp g$. The maximality of $f_{{}_{\ll,g}}$ follows from the maximality of $\tf_{f_{{}_{\ll,g}}}$.
\end{proof}

\begin{corollary}
Let $f$ and $g$ be representable positive functionals on the ${}^{\ast}$-algebra $\mathscr{A}$. Then $f_{{}_{\ll,g}}$ is an extremal point of the convex set of those representable functionals that are majorized by $f$.
\end{corollary}


\begin{thebibliography}{14}

\bibitem{andersonshort}
\newblock{Anderson, Jr., W. N.},
    \newblock{\em Shorted operators},
    \newblock{SIAM J. Appl. Math.}, {20} {(1971)}, {520--525.}

\bibitem{anderson}
    \newblock{Anderson, Jr., W. N. and Trapp, G. E.},
    \newblock{\em Shorted operators. II.},
    \newblock{SIAM J. Appl. Math.}, {28} {(1975)}, {60--71.}

\bibitem{andolebdec}
    \newblock{Ando,~T.},
    \newblock{\em Lebesgue-type decomposition positive operators},
    \newblock{Acta. Sci. Math. (Szeged)}, {38} {(1976)}, {253--260.}

\bibitem{darst}
\newblock{Darst,~R.B.},
\newblock{\em A decomposition of finitely additive set functions},
\newblock{J. for Angew. Math.}, {210} {(1962)}, {31--37.}

\bibitem{douglas}
\newblock{Douglas,~R.G.},
\newblock{\em On majorization, factorization, and range inclusion of operators on Hilbert space},
\newblock{Proc. Amer. Math. Soc.}, {17} {(1996)}, {413--416.}


\bibitem{el}
    \newblock{Eriksson, S.L. and Leutwiler, H.},
    \newblock{\em A potential theoretic approach to parallel addition},
    \newblock{Math. Ann.}, {274} {(1986)}, {301--317.}

\bibitem{gudder}
\newblock{Gudder,~S.},
\newblock{\em A Radon-Nikodym theorem for ${}^{\ast}$-algebras},
\newblock{Pacific J. Math.}, {80(1)}{(1979)},{141--149.}


\bibitem{lebdec}
\newblock{Hassi,~S. and Sebesty{\'e}n,~Z. and de Snoo,~H.},
\newblock{\em Lebesgue type decompositions for nonnegative forms},
\newblock{{J. Funct. Anal.}, {257}{(12)} {(2009)}, {3858--3894.}}

\bibitem{inoue}
\newblock{Inoue,~A.},
\newblock{\em A Radon-Nikodym theorem for positive linear functionals on ${}^{\ast}$-algebras},
\newblock{J. Operator Theory}, {10}{(1983)}, {77--86.}

\bibitem{kosaki}
\newblock{Kosaki,~H.},
\newblock{\em Lebesgue decomposition of states on a von Neumann algebra},
\newblock{American Journal of Math.},
\newblock{Vol 107.} {No.3(1985)}
\newblock{697--735.}

\bibitem{krein}
\newblock{Krein,~M.G.},
\newblock{\em The theory of self-adjoint extensions of semi-bounded Hermitian operators},
\newblock{[Mat. Sbornik] 10 (1947), 431–-495.}

\bibitem{riesz}
\newblock{Riesz, F.},
\newblock{\em Sur quelques notions fondamentales dans la th\'eorie g\'en\'erale des op\'erations lin\'eaires},
\newblock{Ann. of Math.}, {(2) 41} {(1940)}, {174--206.}

\bibitem{rosenberg}
\newblock{Rosenberg,~M.},
\newblock{\em Range decomposition and generalized inverse of nonnegative Hermitian matrices},
\newblock{SIAM Rev.}, {11 (1969)}, {568--571.}

\bibitem{sebestyen}
\newblock{Sebesty\'en,~Z.},
\newblock{\em On representability of linear functionals on ${}^\ast$-algebras},
\newblock{Periodica Math. Hung.}, {15(3)} {(1984)}, {233--239.}

\bibitem{Sebestyén93}
\newblock Sebesty{\'e}n, Z.,
\newblock {\em Operator extensions on {H}ilbert space,}
\newblock Acta Sci. Math. (Szeged), 57 (1993), 233--248.

\bibitem{stt}
\newblock{Sebesty\'en,~Z. and Tarcsay,~Zs. and Titkos,~T.},
\newblock{\em Lebesgue decomposition theorems},
\newblock{Acta Sci. Math. (Szeged)}, {79(1-2)} {(2013)}, {219--233.}

\bibitem{cof}
\newblock{Sebesty\'en,~Z. and Titkos,~T.},
\newblock{\em Complement of forms},
\newblock{Positivity}, {17} {(2013)}, {1--15.}

\bibitem{simon}
\newblock{Simon,~B.},
\newblock{\em A canonical decomposition for quadratic forms with applications to monotone convergence theorems},
\newblock{J. Funct. Anal.} {28} {(1978)}, {377--385.}

\bibitem{sing}
\newblock{Sz\H{u}cs,~Zs.},
\newblock{\em The singularity of positive linear functionals},
\newblock{Acta Math. Hung.},
\newblock{136 (1-2)} {(2012)}, {138-–155.}

\bibitem{jot}
\newblock{Sz\H{u}cs,~Zs.},
\newblock {\em On the Lebesgue decomposition of representable forms over algebras},
\newblock{J. Operator Theory},
\newblock{70:1(2013)}, {3-–31.}

\bibitem{szucs}
\newblock{Sz\H{u}cs,~Zs.},
\newblock {\em On the Lebesgue decomposition of positive linear functionals},
\newblock{Proc. Amer. Math. Soc.}, {141}{(2013)}{619--623.}

\bibitem{tarcsaylebdec}
\newblock{Tarcsay,~Zs.},
\newblock{\em A functional analytic proof of the Lebesgue-Darst decomposition theorem},
{Real Analysis Exchange}, \newblock {39(1), 2013/2014, 241--248.}

\bibitem{tarcsayopdec}
\newblock{Tarcsay,~Zs.},
\newblock{\em Lebesgue-type decomposition of positive operators},
{Positivity}, \newblock {Vol. 17(2013), 803--817.}

\bibitem{tpf-arxiv}
\newblock{Tarcsay,~Zs.},
\newblock{\em Lebesgue decomposition of representable functional on ${}^{\ast}$-algebras.},
\newblock{(Manuscript).}

\bibitem{titkos}
\newblock{Titkos,~T.},
\newblock{\em Lebesgue decomposition of contents via nonnegative forms},
\newblock{Acta Math. Hungar.}, {140(1–-2)} {(2013)}, {151-–161.}

\bibitem{weidmann}
\newblock{Weidmann,~J.},
\newblock{\em Lineare operatoren in Hilbertr\"aumen},
\newblock{B.G. Teubner, Stuttgart}{(1976).}
\end{thebibliography}
\end{document}